\documentclass[reqno]{amsart}
\usepackage{amsfonts,amsmath,amsthm,amssymb,CJK,enumerate}

\usepackage{multirow}
\usepackage{booktabs}

\newtheorem{thm}{Theorem}[section]

\newtheorem{lem}[thm]{Lemma}

\newtheorem{defn}[thm]{Definition}

\theoremstyle{remark}
\newtheorem{rem}[thm]{Remark}

\numberwithin{equation}{section}

\newcommand{\al}{\alpha}

\def\lz{\lambda}
\def\Lz{\Lambda}

\def\({\Bigl(}
\def \){ \Bigr)}

\def\sub{\substack}

 \def\lz{{\lambda}}

 \def\RR{{\mathbb R}}

\def\ss{{\Bbb S}^{d}}

\def\BB{\Bbb B^d}

\def\Lz{\Lambda}

\def\Lz{\Lambda}

\begin{document}
\def\RR{\mathbb{R}}
\def\Exp{\text{Exp}}
\def\FF{\mathcal{F}_\al}

\title[] {Weighted $\ell_q$ approximation problems on the ball and on the sphere}

\author[]{Jiansong Li} \address{ School of Mathematical Sciences, Capital Normal
University, Beijing 100048,
 China}
\email{2210501007@cnu.edu.cn}

\author[]{Heping Wang} \address{ School of Mathematical Sciences, Capital Normal
University, Beijing 100048,
 China}
\email{wanghp@cnu.edu.cn}



\begin{abstract} Let $L_{q,\mu},\, 1\le q<\infty, \ \mu\ge0,$ denote the weighted $L_q$ space with the classical Jacobi weight $w_\mu$ on the ball $\BB$.
We consider the weighted least $\ell_q$ approximation problem for
 a given $L_{q,\mu}$-Marcinkiewicz-Zygmund family  on $\BB$. We
 obtain the weighted least $\ell_q$ approximation errors  for
 the weighted Sobolev space $W_{q,\mu}^r$, $r>(d+2\mu)/q$,  which are order optimal.  We also discuss the least squares quadrature
induced by an $L_{2,\mu}$-Marcinkiewicz-Zygmund family, and get
the quadrature errors  for $W_{2,\mu}^r$, $r>(d+2\mu)/2$, which
are also order optimal. Meanwhile, we give the corresponding the
weighted least $\ell_q$ approximation theorem and the least
squares quadrature errors on the sphere.
\end{abstract}

\maketitle
\input amssym.def

\section{Introduction }

\

Let $\mathbb{B}^{d}=\{x\in \mathbb{R}^d:\,|x|\leq 1\}$ denote the
unit ball of $\mathbb{R}^d$, where $x\cdot y$ is the usual inner
product and $|x|=(x\cdot x)^{1/2}$ is the usual Euclidean norm.
Let $\Pi_n^{d}$ be the space of all polynomials in $d$ variables
of total degree at most $n$. For the classical Jacobi weight on
$\mathbb{B}^{d}$ $$w_\mu(x)=b_{d}^\mu(1-|x|^2)^{\mu-1/2},\ \
\mu\ge0,\ \  b_{d}^\mu=\Big(\int_{\Bbb
B^d}(1-|x|^2)^{\mu-1/2}dx\Big)^{-1},\
$$ we denote by $L_{p,\mu}\equiv L_p(\BB,w_\mu(x)dx), \ 0<
p< \infty,$
 the space of all Lebesgue measurable functions $f$ on
$\mathbb{B}^{d}$ with finite quasi-norm
$$\|f\|_{p,\mu}:=\Big(\int_{\mathbb{B}^{d}}|f(x)|^pw_\mu(x)dx\Big)^{1/p}.$$ When $p=\infty$ we consider the space of
continuous functions $C(\Bbb B^{d})$ with the uniform norm. In
particular, $L_{2,\mu}$ is a Hilbert space with inner product
$$\langle f,g\rangle_\mu:=\int_{\Bbb B^d}f(x)g(x)w_\mu(x)dx,\ {\rm
for}\ f,g\in L_{2,\mu}.$$

 For  $n\in \Bbb N$ we
take $l_n$ points in $\Bbb B^d$ and  $l_n$ positive numbers
$$\mathcal{X}_n:=\{x_{n,k}:\, k=1,2,\dots,l_n\}\ {\rm and}\  \tau_n:=\{\tau_{n,k}:\,k=1,2,\dots,l_n\},$$ and assume
that $l_n\to\infty$ as $n\to\infty$. This  yields a family
$(\mathcal{X},\tau):=(\mathcal{X}_n, \tau_n)_{n\ge1}$. We assume
that the family $(\mathcal{X},\tau)$ constitutes a
Marcinkiewicz-Zygmund (MZ) family on $\BB$ defined as follows.

\begin{defn}\label{def1.1}
Suppose that
$\mathcal{X}=\{\mathcal{X}_n\}=\{x_{n,k}:\,k=1,2,\dots,l_n,\,n=1,2,\dots\}$
is a doubly-indexed set of points in $\Bbb B^d$, and
 $\tau=\{\tau_n\}=\{\tau_{n,k}:\,k=1,2,\dots,l_n,\,n=1,2,\dots\}$ is a doubly-indexed set of positive numbers.  Then for $0< q<\infty$, the family $(\mathcal{X},\tau)$ is
 called an $L_{q,\mu}$-Marcinkiewicz-Zygmund family, denoted by $L_{q,\mu}$-MZ, if there exist constants $A,\, B >0$ independent
of $n$ such that
\begin{equation}\label{1.1}
  A\|P\|_{q,\mu}^q\le\sum\limits_{k=1}^{l_n}|P(x_{n,k})|^q\tau_{n,k}\le B\|P\|_{q,\mu}^q,\ {\rm for\ all}\ P\in \Pi_n^d.
\end{equation}

The ratio $\kappa=B/A$ is the global condition number of
$L_{q,\mu}$-MZ family $(\mathcal{X},\tau)$, and
$\mathcal{X}_n=\{x_{n,k}:\,k=1,2,\dots,l_n\}$ is the $n$-th layer of
$\mathcal{X}$.
\end{defn}
\begin{rem}\label{rem1.1}For $q=\infty$, we say that the family $(\mathcal{X},\tau)$ is $L_{\infty}$-MZ if there exist constants $A>0$ independent
of $n$ such that $$
  A\|P\|_{\infty}\le\max\limits_{1\le k\le l_n}|P(x_{n,k})|\le \|P\|_\infty\ {\rm for\ all}\ P\in \Pi_n^d.
$$The global condition number of
$L_{\infty}$-MZ family is $1/A$.\end{rem}

It follows from \cite{DaX,PX} that such $L_{q,\mu}$-MZ families on
$\BB$ exist. Necessary density conditions for $L_{2,\mu}$-MZ
families on $\BB$ were obtained in \cite{AMO,BO}. There are many
papers devoted to studying MZ families on the sphere and compact
manifold (see \cite{FM,M, MP, OP, W1}).

\begin{rem}\label{rem1.2} It follows from \eqref{1.1} that if $P\in\Pi_n^d$ and $P(x_{n,k})=0,\ k=1,\dots,l_n$, then $P=0$. This means that
usually $\mathcal X_n$ contains more than $\dim \Pi_n^d$ points,
so that it is not an interpolating set for $\Pi_n^d$. We set
\begin{equation}\label{4.1}
  \mu_n:=\sum\limits_{k=1}^{l_n}\tau_{n,k}\delta_{x_{n,k}},
\end{equation}where $\delta_z(g)=g(z)$ for a function $g$ is the evaluation operator.
For any $f\in C(\Bbb B^d)$, we define for $0<q<\infty$,
$$\|f\|_{(q)}:=\Big(\int_{\Bbb
B^d}|f(x)|^qd\mu_n(x)\Big)^{1/q}=\Big(\sum\limits_{k=1}^{l_n}|f(x_{n,k})|^q\tau_{n,k}\Big)^{1/q},$$
and for  $q=\infty$,  $$\|f\|_{(\infty)}:=\max\limits_{1\le k\le
l_n}|f(x_{n,k})|.$$ Hence, for $0<q\le \infty$,
$(\Pi_n^d,\|\cdot\|_{(q)})$ is a Frechet space. It follows from
\eqref{1.1} that the $L_{q,\mu}$-norm of a polynomial of degree at
most $n$ on $\Bbb B^d$ is comparable to the discrete version given
by the weighted $\ell_q$-norm of its restriction to
$\mathcal{X}_n$.
\end{rem}

This paper is concerned with constructive polynomial approximation
on $\BB$ which uses function values (the samples) at the  points
in $\mathcal{X}_n$ (sometimes called standard information).
 For an $L_{q,\mu}$-MZ family $(\mathcal{X},\tau)$, we
usually sample a continuous function $f$ on the $n$-th layer
$\mathcal{X}_n$ and apply the samples to construct an
approximation to $f$. We use the weighted least $\ell_q$
algorithms.  This means that our problem is to solve a sequence of
weighted least $\ell_q$ approximation problems with samples taken
from the samples $\mathcal{X}_n$. We recall the following
definition.
\begin{defn}\label{def1.3}
Let $0< q\le\infty$, and let $(\mathcal{X},\tau)$ be an
$L_{q,\mu}$-MZ family. For $f\in C(\Bbb B^d)$, we define the
weighted least $\ell_q$ approximation by
\begin{align}\label{1.2}
         L_{n,q}(f)&:=\arg\min\limits_{P\in\Pi_n^d}\,\Big(\sum\limits_{k=1}^{l_n}|f(x_{n,k})-P(x_{n,k})|^q\tau_{n,k}\Big)^{1/q}.
 \end{align}That is,  $L_{n,q}(f)$ is any function in $\Pi_n^d$ satisfying
 $$\|f-L_{n,q}(f)\|_{(q)}=\min_{P\in\Pi_n^d}\|f-P\|_{(q)}.$$
\end{defn}

\begin{rem}\label{rem1.4}
Clearly, for $f\in C(\Bbb B^d)$ and $0<q<\infty$, the minimizer
$L_{n,q}(f)$ exists. Hence, this definition is well defined. For
$1< q<\infty$, $L_{n,q}(f)$ is unique. However, if $0<q\le1$ or
$q=\infty$, then $L_{n,q}(f)$ may be not unique.
\end{rem}

\begin{rem}\label{rem1.5}
If $q\neq 2$, then  $L_{n,q}$ is not linear. That is, there exist
$f_1,f_2\in C(\Bbb B^d)$ such that
$$L_{n,q}(f_1+f_2)\neq L_{n,q}(f_1)+L_{n,q}(f_2).$$ However,
 for  $f\in C(\Bbb B^d)$ and $P_n\in\Pi_n^d$,  by the definition of $L_{n,q}(f)$ we have
\begin{align}\label{1.3}
 \|f-L_{n,q}(f)\|_{(q)}&=\min\limits_{P\in\Pi_n^d}\|f-P\|_{(q)}\notag \\ &= \|(f+P_n)-L_{n,q}(f+P_n)\|_{(q)}\notag\\&\le \|f+P_n\|_{(q)}.
\end{align}
\end{rem}

The most interesting  case is $q=2$ in which $L_{n,2}$ has many
good properties. Indeed, $L_{n,2}$ is a bounded linear operator on
$C(\BB)$ satisfying that $L_{n,2}^2=L_{n,2}$, and the range of
$L_{n,2}$ is $\Pi_n^d$. If we define the discretized inner product
on $C(\BB)$ by
$$\langle
f,g\rangle_{(2)}:=\sum\limits_{k=1}^{l_n}\tau_{n,k}f(x_{n,k})g(x_{n,k}),$$
then $L_{n,2}$ is just the orthogonal projection onto $\Pi_n^d$
with respect to the discretized inner product $\langle
\cdot,\cdot\rangle_{(2)}$. Hence, we obtain  for $f\in C(\Bbb
B^d)$,
\begin{equation*}
  L_{n,2}(f)(x)=\langle f,
  D_n(x,\cdot)\rangle_{(2)}=\sum_{k=1}^{l_n}\tau_{n,k}f(x_{n,k})D_n(x,x_{n,k}),
\end{equation*}where $D_n(x,y)$ is the reproducing kernel of $\Pi_n^d$ with respect to the discretized inner product $\langle\cdot,\cdot\rangle_{(2)}$.
We call  $L_{n,2}(f)$   the weighted least squares polynomial, and
$L_{n,2}$ the weighted least squares operator.

Following Gr\"ochenig  in \cite{G},  for $L_{2,\mu}$-MZ family we can
use the frame theory to construct the  quadrature on $\BB$
$$I_n(f)=\sum_{k=1}^{l_n}w_{n,k}f(x_{n,k}).$$ It was shown in
\cite{LW} that
$$w_{n,k}=\tau_{n,k}\int_{\BB}D_n(x,x_{n,k})w_\mu(x)dx, $$and
\begin{equation}\label{1.5}I_n(f)=\int_{\BB}L_{n,2}(f)(x)w_\mu(x)dx.\end{equation}Such
quadrature $I_n$ is called the  least squares  quadrature.

Given an $L_2$-MZ family $(\mathcal{X},\tau)$ on a usual compact
space $\mathcal{M}$ with some structure, Gr\"ochenig in \cite{G}
studied the weighted least squares  approximation
$L_{n,2}^{\mathcal M}(f)$ and the least squares quadrature
$I_n^{\mathcal M}$ from the samples $\mathcal{X}_n$ of a function
$f$, and obtained the following approximation theorems and
quadrature errors as follows. For $r>d/2$, we have
\begin{equation}\label{1.6}
  \|f-L_{n,2}^{\mathcal M}(f)\|_{2}\le
c(1+\kappa^2)^{1/2}n^{-r+d/2}\|f\|_{H^r(\mathcal{M})},
\end{equation}
and
\begin{equation}\label{1.7}
  \Big|\int_{\mathcal M}f(x)d\mu(x)-I_n^{\mathcal M}(f)\Big|\le
c(1+\kappa^{1/2})n^{-r+d/2}\|f\|_{H^r(\mathcal{M})},
\end{equation}
where $c$ depends on $d$, $r$, but not on $f$, $\kappa$ or
$(\mathcal{X},\tau)$, $\mu$ is a probability measure on $\mathcal
M$, and $H^r(\mathcal{M})$ is the Sobolev space on $\mathcal{M}$
(see \cite{G}). However, the obtained error estimates are not
optimal due to the generality of $\mathcal{M}$. Lu and Wang in
\cite{LW} investigated the weighted least squares approximation
$L_{n,2}^{\Bbb S}$  and the least squares quadrature $I_{n}^{\Bbb
S}$  on the sphere $\Bbb S^d$, and obtained the following optimal
error estimates. For $r>d/2$,  we have
\begin{equation}\label{1.7-0} \|f-L_{n,2}^{\Bbb S}(f)\|_{2}\le
c(1+\kappa^2)^{1/2}n^{-r}\|f\|_{H^r(\Bbb S^d)},\end{equation} and
\begin{equation}\label{1.7-1}\Big|\int_{\ss}f(x)d\sigma (x)-I_n^{\Bbb S}(f)\Big|\le
c(1+\kappa^2)^{1/2}n^{-r}\|f\|_{H^r(\ss)},\end{equation} where $c$
depends only on $d$ and  $r$,  and $H^r(\Bbb S^d)$ is the Sobolev
space on $\Bbb S^d$.

 Gr\"ochenig commented in \cite{G} that  Marcinkiewicz-Zygmund families with respect to general
$q$-norms seemed to require different techniques. In this paper we
consider $L_{q,\mu}$-MZ families  on $\BB$ for $1\le q<\infty$. We
use the weighted least $\ell_q$ approximation to obtain the
optimal approximation errors of the weighted Sobolev classes
$BW_{q,\mu}^r$ on $\BB$ (see Section 2 for definition of
$BW_{q,\mu}^r$). The techniques we used are different from the
ones in \cite{G} even in the case $q=2$.
 We remark that $W_{2,\mu}^r$ is just
the Sobolev space $H^r(\BB)$ given in \cite{G}, and if
$r>(d+2\mu)/q$, then the weighted Sobolev space $W_{q,\mu}^r$ can
be compactly embedded into $C(\BB)$.  Our main results can be
formulated as follows.

\begin{thm}\label{thm1.6} Let $1\le p\le \infty, \,1\le q<\infty$, and $\mu\ge 0$. Suppose that $(\mathcal{X},\tau)$
is an $L_{q,\mu}$-MZ family with global condition number
$\kappa=B/A$, $L_{n,q}$ is the weighted least $\ell_q$
approximation  defined by \eqref{1.2}. If $f\in W_{p,\mu}^r$,
$r>(d+2\mu)\max\{1/p,1/q\}$, then we have
\begin{equation}\label{1.8}
  \|f-L_{n,q}(f)\|_{q,\mu}\le C(1+\kappa^{1/q})n^{-r+(d+2\mu)(1/p-1/q)_+}\|f\|_{W_{p,\mu}^r},
\end{equation}where $C>0$ are independent of  $f$, $n$,
$\kappa$, and  $(\mathcal{X},\tau)$, and  $a_+=\max\{a,0\}$.
\end{thm}

The following theorem   follows from Theorem \ref{thm1.6} and
\eqref{1.5} immediately.

\begin{thm}\label{thm1.7} Let $\mu\ge 0$. Suppose that $(\mathcal{X},\tau)$
is an $L_{2,\mu}$-MZ family with global condition number
$\kappa=B/A$, $L_{n,2}$ and $I_n$ are the weighted least squares
approximation   and the  least squares quadrature, respectively.
If $f\in H^r(\BB)\equiv W_{2,\mu}^r$, $r>(d+2\mu)/2$, then we have
\begin{equation}\label{1.9}
  \|f-L_{n,2}(f)\|_{2,\mu}\le C(1+\kappa^{1/2})n^{-r}\|f\|_{H^r(\BB)},
\end{equation}and
\begin{equation}\label{1.10}
  \Big|\int_{\BB}f(x)w_\mu(x)dx-I_n(f)\Big|\le C(1+\kappa^{1/2})n^{-r}\|f\|_{H^r(\BB)},
\end{equation}
where $C>0$ are independent of  $f$, $n$, $\kappa$, and
$(\mathcal{X},\tau)$.
\end{thm}

We also give the corresponding weighted least $\ell_q$
approximation and least squares quadrature results on $\Bbb S^d$.

The contribution of this paper contains three aspects. First, we
obtain the corresponding results for Marcinkiewicz-Zygmund
families with respect to general $q$-norms by a different method
from \cite{G}. Second, the obtained error estimates \eqref{1.8}
for $1\le q= p<\infty$ and \eqref{1.10} are asymptotically optimal
 in a variety of
Sobolev space settings (as explained in Remark \ref{rem1.8}
below). Third, we reduce dependence on the global condition number
in  \eqref{1.9}  by replacing the constant $(1+\kappa^2)^{1/2}$ in
\eqref{1.6} with  the constant $1+\kappa^{1/2}$.

Throughout the paper, the notation $a_{n}\asymp b_{n}$ means $a_{n}\lesssim b_{n}$  and $a_{n}\gtrsim b_{n}$. Here, $a_{n}\lesssim b_{n}\,(a_{n}\gtrsim b_{n})$ means that there exists a constant $c>0$ independent of $n$ such that $a_{n}\leq c b_{n}\,(b_{n}\leq c a_{n})$.

\begin{rem}\label{rem1.8} Let $F$ be  a class of continuous functions on $D$, and $(X,\|\cdot\|_X)$ be a normed linear space of functions on
$D$, where $D$ is a subset of $\Bbb R^d$, $\nu$ is a probability
measure on $D$. For $N\in \Bbb N$, the sampling numbers (or the
optimal recovery) of $F$ in $X$ are defined by
$$g_N(F,X) \ :=\inf_{\sub{\xi_1,\dots,\xi_N\in D\\ \varphi:\ \Bbb R^N\rightarrow X}}
\sup_{f\in F}\|f-\varphi(f(\xi_1),\dots,f(\xi_N))\|_X,
$$where the infimum is taken over all $N$ points
$\xi_1,\dots,\xi_N$ in $D$ and all  mappings $\varphi$ from $\Bbb
R^N$ to $X$. And the optimal quadrature errors of $F$ are defined
by
$$e_N(F;{\rm INT}):=\inf_{\sub{\xi_1,\dots,\xi_N\in D\\ \lz_1,\dots,\lz_N\in\Bbb R}}
\sup_{f\in
F}\Big|\int_{D}f(x)d\nu(x)-\sum\limits_{j=1}^N\lambda_jf(\xi_j)\Big|,
$$ where the infimum is taken over all $N$ points
$\xi_1,\dots,\xi_N$ in $D$ and all $N$  numbers
$\lz_1,\dots,\lz_N$.

It follows from \cite{PX} that there exist $L_{q,\mu}$-MZ families
on $\BB$ with $l_n\asymp N\asymp n^d$. Combining with
\cite[Theorem 3.5]{LW2}, for such $L_{q,\mu}$-MZ family, and $1\le
q\le p<\infty$, $r>(d+2\mu)/q$, we obtain,
\begin{align*}\sup_{f\in BW_{p,\mu}^{r}}\|f-L_{n,q}(f)\|_{q,\mu}\asymp N^{-r/d}\asymp g_N(BW_{p,\mu }^{r},L_{q,\mu}),\end{align*}
which implies that
 the weighted least $\ell_q$ approximation operators $L_{n,q}$ are asymptotically optimal algorithms in the sense of optimal recovery.

Also, for the least squares quadrature rules $I_n$, it follows
from \eqref{1.10} and \cite[Theorem 1.1]{LW2} that for
$r>(d+2\mu)/2$,
$$ \sup\limits_{f\in BH^{r}(\BB)}\Big|\int_{\BB}f(x)w_\mu(x)dx-I_n(f)\Big|\asymp N^{-r/d}\asymp e_N(BH^{r}(\BB);{\rm INT}),$$which means that the  least squares quadrature
rules  $I_n$ are the asymptotically optimal quadrature formulas
for $BH^{r}(\BB)$.
\end{rem}

The outline of this paper is as follows. In the next section, we
recall some basic results about harmonic analysis on the ball,
introduce the filtered approximation and prove some axillary
lemmas. In Section 3 we give the proof of Theorem \ref{thm1.6}.
Finally, in Section 4 we give the corresponding  weighted least
$\ell_q$ approximation and least squares quadrature results on
$\Bbb S^d$.

\section{Preliminaries}

\subsection{Harmonic analysis on the ball }

\

Let $\mathbb{B}^{d}$ denote the unit ball of $\mathbb{R}^d$, and
$L_{p,\mu}$,  $1\le p< \infty,$
 the weighted $L_p$ space on
$\mathbb{B}^{d}$.
We denote  by $\mathcal{V}_n^d(w_\mu)$ the space of all  polynomials of degree $n$
which are orthogonal to lower degree polynomials in $L_{2,\mu}$, and   $$a_n^d:={\rm dim}\,\mathcal{V}_n^d(w_\mu)=\binom{n+d-1}{n}\asymp n^{d-1}.$$
 It is well known (see \cite[p.38 or p.229]{DX}) that
 $$D_\mu P=-n(n+2\mu+d-1)P,\ \ {\rm for\ all}\ P\in\mathcal{V}_n^d(w_\mu),$$where the second-order differential
operator
$$
 D_\mu:=\triangle-(x\cdot\nabla)^2-(2\mu+d-1)x\cdot \nabla,
$$
 and $\triangle$, $\nabla$ are the Laplace operator, the gradient operator, respectively.

The standard Hilbert theory states that the spaces
$\mathcal{V}_n^d(w_\mu)$ are mutually orthogonal in $L_{2,\mu}$.
Let $\{\phi_{nk}\equiv\phi_{nk}^d: k=1,2,\dots,a_n^d\}$ be a fixed
orthonormal basis for $\mathcal{V}_n^d(w_\mu)$. Then
$$\{\phi_{nk}: k=1,2,\dots,a_n^d,\ n=0,1,2,\dots\}$$ is an
orthonormal basis for $L_{2,\mu}$.

  The orthogonal projector ${\rm Proj}_n :\, L_{2,\mu}\rightarrow \mathcal{V}_n^d(w_\mu)$ can be written
 as\begin{align*}
     ({\rm Proj}_n f)(x)&=\sum\limits_{k=1}^{a_n^d}\langle f, \phi_{nk}\rangle\phi_{nk}(x)
     =\langle f,P_n(w_\mu; x,\cdot)\rangle_\mu,\ \ x\in\BB,
   \end{align*}
 where $P_n(w_\mu; x, y)=\sum\limits_{k=1}^{a_n^d}\phi_{nk}(x)\phi_{nk}(y)$ is the reproducing kernel of $\mathcal{V}_n^d(w_\mu)$.
See \cite{DX} for more details about $P_n(w_\mu; x, y)$.

 Given $r>0$, define the fractional power $(-D_\mu)^{r/2}$ of the operator $-D_\mu$ on $f$ by
$$(-D_\mu)^{r/2}(f):=\sum\limits_{k=0}^\infty (k(k+2\mu+d-1))^{r/2}{\rm Proj}_kf,$$in the sense of distribution.
Using this operator we define the weighted Sobolev space as follows: for $r>0$ and $1\le p\le\infty$,
$$W_{p,\mu}^r:=\{f\in L_{p,\mu}:\|f\|_{W_{p,\mu}^r}:=\|f\|_{p,\mu}+\|(-D_\mu)^{r/2}(f)\|_{p,\mu}<\infty\},$$
while the weighted Sobolev class $BW_{p,\mu}^r$ is defined to be
the unit ball of the weighted Sobolev space $W_{p,\mu}^r$. Note
that if $r>(d+2\mu)/p$, then $W_{p,\mu}^r$ is compactly embedded
into $C(\mathbb{B}^{d})$.

We introduce a  metric $d$ on $\mathbb{B}^d$ by
$$ d(x,y):= \arccos \Big( (x,y)+\sqrt{1-|x|^2}\sqrt{1-|y|^2}\Big),\ \ x,\, y\in\BB.$$
For $r>0,\ x\in \mathbb{B}^d$ and a positive integer $n$,  we set
$${\bf B}(x,r):=\{y\in \mathbb{B}^d\ |  \ d(x,y)\le r\}\ \ \ {\rm and} \ \  \
 \mathcal{W}_\mu(n;x):=(\sqrt{1-|x|^2}+n^{-1})^{2\mu}.$$
It follows from \cite[Equation (4.23)]{PX} that
\begin{equation}\label{2.1}
\frac
{1}{2^\mu(1+nd(x,y))^{2\mu}}\leq\frac{\mathcal{W_\mu}(n;x)}{\mathcal{W_\mu}(n;y)}\leq{2^\mu(1+nd(x,y))^{2\mu}},\
x, y\in\mathbb{B}^{d}.
\end{equation}

For $\epsilon>0$, we say that a finite subset $\Lambda\subset\mathbb{B}^d$ is maximal $\epsilon$-separated if
$$\mathbb{B}^d\subset\bigcup\limits_{y\in\Lambda}{\bf B}(y,\epsilon) \ {\rm\ and}\ \min\limits_{y\neq y^\prime}d(y,y^\prime)\ge \epsilon.$$
Such a maximal $\epsilon$-separated set $\Lambda$ exists  and
satisfies (see \cite[Lemma 5.2]{PX})
\begin{equation}\label{2.2}
  1\le\sum\limits_{\xi\in\Lambda}\chi_{{\bf B}(\xi,\epsilon)}(x)\le C_d,\ {\rm for\ any}\ x\in\Bbb B^d.
\end{equation}

\subsection{The filtered approximation}

\

Now we introduce the filtered approximation on the ball as in
\cite{PX}. In the filtered approximation, the terms in the Fourier
series are to be modified by multiplication by $\eta (k/n)$, where
$\eta \in C^\infty([0,\infty))$ is a ``$C^\infty$-filter''
satisfying
$$\chi_{\left[0,1\right ]}\leq\eta\leq\chi_{\left[0,2\right]}.$$Here, $\chi_A$ denotes the characteristic function of $A$ for
$A\subset \Bbb R$.
The filtered approximation of $f$ is defined by
\begin{align}
V_{n}(f)(x)&:= \sum_{k=0}^{\infty}\eta (\frac kn){\rm
Proj}_{k}(f)(x)=\langle f,
K_{n,\eta}(x,\cdot)\rangle_\mu,\label{2.3}
\end{align}where \begin{equation*}K_{n,\eta}(x,y)=\sum_{k=0}^\infty \eta(\frac kn)P_k(w_\mu; x,y)
=\sum_{k=0}^{2n-1} \eta(\frac kn)P_k(w_\mu; x,y).\end{equation*}
Evidently, $V_n(f)\in\Pi_{2n-1}^d$,  and for any $f\in\Pi_n^d$ we
have
\begin{equation}\label{2.4}f(x)=\int_{\Bbb B^d}f(y)K_{n,\eta}(w_\mu;x,y)w_\mu(y)dy,\ x\in
 \Bbb B^d.\end{equation}

It follows from \cite[Theorem 4.2]{PX} that for $x,y\in\Bbb B^d$,
\begin{equation}\label{2.5}
  |K_{n,\eta}(x,y)|\lesssim
  \frac{n^d}{\sqrt{\mathcal{W}_\mu(n;x)}\sqrt{\mathcal{W}_\mu(n;y)}(1+nd(x,y))^{d+\mu+1}}=:g_n(x,y).
\end{equation}By  \eqref{2.5} and \cite[Lemma
4.6]{PX} we have
\begin{equation}\label{2.5-0}\max_{y\in\BB}\int_{\BB}|K_{n,\eta}(x,y)|w_\mu(x)dx\lesssim
\max_{y\in\BB}\int_{\BB}g_n(x,y)w_\mu(x)dx\lesssim
1.\end{equation}

For $n\in\mathbb{N}$ we define
$$E_n(f)_{p,\mu}:=\inf\{\|f-P\|_{p,\mu}:P\in\Pi_n^d\}
.$$ It is well known (see \cite{PX, Xu}) that for $1\le
p\le\infty$, $r>0$, and $f\in W_{p,\mu}^r$,
\begin{equation}\label{2.6}
  \|f-V_n(f)\|_{p,\mu}\lesssim E_n(f)_{p,\mu}\lesssim n^{-r}\|f\|_{W_{p,\mu}^{r}}.
\end{equation}

\subsection{Some auxiliary lemmas}

\

In this subsection, we give some  lemmas which will be needed  in
the next section.

Let $n$ be a positive integer. Suppose that $\Omega_n$ is a finite
set of $\mathbb{B}^d$, and $\Gamma_n:=\{\mu _{\omega
}:\,\omega\in\Omega_n\}$ is a set of positive numbers.
 The induced measure $\lambda_n$ by $\Gamma_n$ is defined by
\begin{equation}\label{2.7}
\lambda_n:=\sum\limits_{\omega\in\Omega_n}\mu_{\omega}\delta_\omega,
\end{equation}where $\delta_z(f)=f(z)$ for a function $f$  is the evaluation operator.
Hence, for any continuous function $f$ defined on $\mathbb{B}^d$,
we have
$$\int_{\mathbb{B}^d}f(x)d\lambda_n(x)=\sum\limits_{\omega\in
\Omega_n}\mu_{\omega}f(\omega).$$ We call the induced measure
$\lambda_n$ satisfies the regularity condition with a constant $N$
if the inequality $$\int_{{\bf B}(y,\frac1n)}d\lambda_n(x)\le N
\int_{{\bf B}(y,\frac1n)}w_\mu(x)dx$$holds.  That is,
\begin{equation}\label{2.8} \sum _{\omega \in
\Omega_n\cap {\bf B}(y,\frac1n)}\mu_{\omega }
 \le N w_\mu({\bf B}(y,\frac1n)),\ {\rm for\ any}\ y\in\Bbb B^d,\end{equation}holds, where $w_\mu(E):=\int_{E} w_\mu(x)dx$ for any measurable $E\subset\Bbb
 B^d$.

\begin{lem}\label{lem2.1}(\cite[Theorem 3.1]{WHLW})
Suppose that $n\in\Bbb N$, $\Omega_n$ is a finite subset of
$\mathbb{B}^d$, and $\Gamma_n:=\{\mu _{\omega
}:\,\omega\in\Omega_n\}$ is a set of positive numbers. If there
exist  $p_{0}\in (0,\infty )$ and $M>0$ such that for any $f\in
\Pi _{n}^{d}$,
\begin{equation}\label{2.9}
  \sum _{\omega \in \Omega_n}\mu_{\omega }| f(\omega )|^{p_{0}}\le
  M
\int_{\mathbb{B}^d}| f(x)|^{p_{0}}w_{\mu }(x)dx,
\end{equation}
then  the following regularity condition
\begin{equation}\label{2.10} \sum _{\omega \in
\Omega_n\cap {\bf B}(y,1/n)}\mu_{\omega }
 \le CM w_\mu({\bf B}(y,1/n)),\ {\rm for\ any}\ y\in\BB,\end{equation} holds,  where $C>0$ depends only on $d,\,\mu$, and  $p_0$.
\end{lem}

\begin{lem}\label{lem2.2}
Suppose that $n\in\Bbb N$, $\Omega_n$ is a finite subset of
$\mathbb{B}^d$, and $\Gamma_n:=\{\mu _{\omega
}\}_{\omega\in\Omega_n}$ is a set of positive numbers. If the
induced measures $\lambda_n$ by $\Gamma_n$ satisfies the
regularity condition \eqref{2.8}  with a constant $N$, then for any $1\le
p<\infty$, $m\in\Bbb N$, $m\geq n$, and $f\in \Pi _{m}^{d}$, we
have
\begin{equation}\label{2.11}\sum_{\omega \in \Omega_n}\mu_{\omega }|f(\omega )|^{p}\le
CN(\frac{m}{n})^{d+2\mu}\int_{\mathbb{B}^d}|f(x)|^{p}w_{\mu}(x)dx,\end{equation}where
$C>0$ depends only on $d,\,\mu$, and  $p$.
\end{lem}
 Lemma \ref{lem2.2} is an improvement of \cite[Corollary 3.3]{WHLW} with the exponent $d+2\mu+1$  in \cite[Corollary 3.3]{WHLW} replaced by $d+2\mu$ in \eqref{2.11}.
   Its proof is based on the following lemma.

\begin{lem}\label{lem2.3} Suppose that $n\in\Bbb N$, $\Omega_n$ is a finite subset
of $\mathbb{B}^d$, and $\Gamma_n:=\{\mu _{\omega
}\}_{\omega\in\Omega_n}$ is a set of positive numbers. If the
induced measure $\lambda_n$ by $\Gamma_n$ satisfies the regularity
condition \eqref{2.8} with a constant $N$, then for any $m\in\Bbb
N$, $m\geq n$, we have
\begin{equation}\label{2.12}\max\limits_{y\in\mathbb{B}^d}\Big(\sum _{\omega \in \Omega_n}\mu_{\omega }|K_{m,\eta}(\omega,y )|\Big)\le
CN(\frac{m}{n})^{d+2\mu},\end{equation}where $C>0$ is independent
of $m, n$ and $N$.
\end{lem}

\begin{proof}
For $x,y\in\mathbb{B}^d$, and $m\ge n$,  set
\begin{equation*}
  g_m(x,y)=\frac{m^d}{\sqrt{\mathcal{W}_\mu(m;x)}\sqrt{\mathcal{W}_\mu(m;y)}(1+md(x,y))^{d+\mu+1}}.
\end{equation*}
Note that
$$\frac{\mathcal{W}_\mu(n;x)}{\mathcal{W}_\mu(m;x)}=\Big(\frac{1/n+\sqrt{1-|x|^2}}{1/m+\sqrt{1-|x|^2}}\Big)^{2\mu} =\Big(\frac{1+n\sqrt{1-|x|^2}}{1+m\sqrt{1-|x|^2}}\Big)^{2\mu}(\frac{m}{n})^{2\mu}\le (\frac{m}{n})^{2\mu}.$$
It follows that
\begin{align*}
 g_m(x,y)&\le(\frac{m}{n})^{2\mu}\frac{m^d}{\sqrt{\mathcal{W}_\mu(n;x)}\sqrt{\mathcal{W}_\mu(n;y)}(1+md(x,y))^{d+\mu+1}}\nonumber
 \\&\le(\frac{m}{n})^{d+2\mu}\frac{n^d}{\sqrt{\mathcal{W}_\mu(n;x)}\sqrt{\mathcal{W}_\mu(n;y)}(1+nd(x,y))^{d+\mu+1}}\nonumber\\&=(\frac{m}{n})^{d+2\mu}g_n(x,y).
\end{align*}
According to \eqref{2.5}, we have
\begin{equation}\label{2.13}|K_{m,\eta}(x,y)|\lesssim g_m(x,y)\le
(\frac{m}{n})^{d+2\mu}g_n(x,y).\end{equation}

Let $\Lambda_n$ be a maximal $\frac{1}n$-separated set on $\BB$.
For $\xi\in \Lz_n$ and $x,x^\prime\in{\bf B}(\xi,\frac 1n)$, by
\eqref{2.1} we have
$$g_n(x',y)\asymp g_n(\xi,y)\asymp g_n(x,y),$$which leads to
\begin{equation}\label{2.14}
  \max\limits_{x\in{\bf B}(\xi,\frac 1n)}g_n(x,y)\lesssim \min\limits_{x\in{\bf B}(\xi,\frac 1n)}g_n(x,y).
\end{equation}
It follows from \eqref{2.13}  that for any $y\in\mathbb{B}^d$,
\begin{align*}
  \sum_{\omega\in\Omega_n}\mu_{\omega}|K_{m,\eta}(\omega,y)|&=\int_{\mathbb{B}^d}|K_{m,\eta}(x,y)|d\lambda_n(x)
\\&\lesssim (\frac{m}{n})^{d+2\mu}\int_{\mathbb{B}^d}g_n(x,y)d\lambda_n(x)
\\&\le (\frac{m}{n})^{d+2\mu}\int_{\mathbb{B}^d}g_n(x,y)\Big(\sum\limits_{\xi\in\Lambda_n}\chi_{{\bf B}(\xi,\frac 1 n)}(x)\Big)d\lambda_n(x)
\\&\le (\frac{m}{n})^{d+2\mu}\sum\limits_{\xi\in\Lambda_n}\max\limits_{x\in{\bf B}(\xi,\frac 1n)}g_n(x,y)\int_{{\bf B}(\xi,\frac 1n)}d\lambda_n(x)
\\&\lesssim N(\frac{m}{n})^{d+2\mu}\sum\limits_{\xi\in\Lambda_n}\min\limits_{x\in{\bf B}(\xi,\frac 1n)}g_n(x,y)\int_{{\bf B}(\xi,\frac 1n)}w_\mu(x)dx
\\&\le N(\frac{m}{n})^{d+2\mu}\sum\limits_{\xi\in\Lambda_n}\int_{\mathbb{B}^d}g_n(x,y)\chi_{{\bf B}(\xi,\frac 1 n)}(x)w_\mu(x)dx
\\&= N(\frac{m}{n})^{d+2\mu}\int_{\mathbb{B}^d}g_n(x,y)\sum\limits_{\xi\in\Lambda_n}\chi_{{\bf B}(\xi,\frac 1n)}(x)w_\mu(x)dx
\\&\lesssim N(\frac{m}{n})^{d+2\mu}\int_{\mathbb{B}^d}g_n(x,y)w_\mu(x)dx\lesssim N(\frac{m}{n})^{d+2\mu},
\end{align*}
where  in the second and the second  last inequalities we used
\eqref{2.2}; in the fourth inequality we used the regularity
condition and \eqref{2.14}; and in the last inequality we used
\eqref{2.5-0}. This completes the proof of Lemma \ref{lem2.3}.
\end{proof}

Now we turn to prove Lemma \ref{lem2.2}.

\

\noindent{\it Proof of Lemma \ref{lem2.2}}

The proof is standard (see \cite[Corollary 3.3]{WHLW}). For the
convenience of the readers we give the proof.

Applying \eqref{2.7} and the H\"older inequality, we have for
$f\in\Pi_m^d$ with $m\ge n$, $1\le p<\infty$,  and
$x\in\mathbb{B}^d$,
\begin{align}\label{2.17}
  |f(x)| & =\Big|\int_{\mathbb{B}^d}f(y)K_{m,\eta}(x,y)w_\mu(y)dy\Big|\nonumber
  \\&\le\int_{\mathbb{B}^d}|f(y)||K_{m,\eta}(x,y)|^{\frac1p}|K_{m,\eta}(x,y)|^{\frac{p-1}p}w_\mu(y)dy\nonumber
  \\&\le\Big(\int_{\mathbb{B}^d}|f(y)|^p|K_{m,\eta}(x,y)|w_\mu(y)dy\Big)^{\frac1p}\Big(\int_{\mathbb{B}^d}|K_{m,\eta}(x,y)|w_\mu(y)dy\Big)^{\frac{p-1}p}\nonumber
  \\&\lesssim \Big(\int_{\mathbb{B}^d}|f(y)|^p|K_{m,\eta}(x,y)|w_\mu(y)dy\Big)^{\frac1p},
\end{align}
where in the last inequality we used \eqref{2.5-0}. It follows
from \eqref{2.17} and Lemma \ref{lem2.3} that
\begin{align*}
  \sum _{\omega \in \Omega_n}\mu_{\omega }| f(\omega )|^{p}
  &\lesssim \sum _{\omega \in \Omega_n}\mu _{\omega }\int_{\mathbb{B}^d}|f(y)|^p|K_{m,\eta}(\omega,y)|w_\mu(y)dy
  \\&=\int_{\mathbb{B}^d}|f(y)|^p\Big(\sum _{\omega \in \Omega_n}\mu_{\omega }|K_{m,\eta}(\omega,y)|\Big)w_\mu(y)dy
 \\& \le \int_{\mathbb{B}^d}|f(y)|^pw_\mu(y)dy\,\max\limits_{y\in\mathbb{B}^d}\Big(\sum _{\omega\in \Omega_n}\mu _{\omega }|K_{m,\eta}(\omega,y)|\Big)
 \\& \lesssim N(\frac{m}{n})^{d+2\mu}\int _{\mathbb{B}^d}| f(x)|^{p}w_{\mu }(x)dx.
\end{align*}
Lemma \ref{lem2.2} is proved. $\hfill\Box$

\

Finally we give the Nikolskii inequalities on $\BB$.

\begin{lem}(\cite[Proposition 2.4]{KPX}) Let $ 1\le p,q\le \infty$ and $\mu\ge 0$. Then for  any $P\in \Pi_n^d$ we
have,
\begin{equation}\label{2.19}\|P\|_{q,\mu}\lesssim n^{(d+2\mu)(1/p-1/q)_+}\|P\|_{p,\mu}.\end{equation}\end{lem}

\

\section{Proof of Theorem \ref{thm1.6} }

In this section, we give the following lemma from which Theorem
\ref{thm1.6} follows immediately.

\begin{lem}\label{lem3.1}Let $1\le p\le\infty$, $1\le t,q<\infty$, and $\mu\ge 0$. Suppose that $(\mathcal{X},\tau)$ is an $L_{t,\mu}$-MZ family
with global condition number $\kappa=B/A$, $L_{n,t}$ is the
weighted least $\ell_t$ approximation  defined by \eqref{1.2}. If
$f\in W_{p,\mu}^r$, $r>(d+2\mu)\max\{1/p,1/t\}$, then we have
\begin{equation}\label{3.1}
  \|f-L_{n,t}(f)\|_{q,\mu}\le C(1+\kappa^{1/t})n^{-r+(d+2\mu)[(\frac1p-\frac1t)_++(\frac1t-\frac1q)_+]}\|f\|_{W_{p,\mu}^r},
\end{equation}where $C>0$  is independent of  $f$, $n$, $\kappa$, and $(\mathcal{X},\tau)$.
\end{lem}

\begin{proof}

For $n\in\Bbb N$, we choose a nonnegative integer $s$ such that
$$2^s\le n<2^{s+1},$$ and for $f\in W_{p,\mu }^r$,  we define
$$\sigma_1(f)=V_{1}(f), \ \ \sigma_j(f)=V_{2^{j-1}}(f)-V_{2^{j-2}}(f),\ \
{\rm for}\ j\ge2.$$ Note that $$\sigma_j(f)\in\Pi_{2^{j}}^d\ {\rm
and}\ V_{2^{s-1}}(f)=\sum\limits_{j=1}^s\sigma_j(f)\in
\Pi_{2^{s}}^d.$$ By \eqref{2.6} we get
\begin{align}\label{3.2}
  \|\sigma_j(f)\|_{p,\mu}&\le \|f-V_{2^{j-1}}(f)\|_{p,\mu}+\|f-V_{2^{j-2}}(f)\|_{p,\mu}\nonumber
  \\&\lesssim E_{2^{j-1}}(f)_{p,\mu}+E_{2^{j-2}}(f)_{p,\mu}\lesssim 2^{-jr}\|f\|_{W_{p,\mu}^r}.
\end{align}

For $f\in W_{p,\mu }^r$, $r>(d+2\mu)/p$, we have
\begin{align}\label{3.3}
  \|f-L_{n,t}(f)\|_{q,\mu}&\le
  \|f-V_{2^{s-1}}(f)\|_{q,\mu}+\|L_{n,t}(f)-V_{2^{s-1}}(f)\|_{q,\mu}.
\end{align}
First we estimate $\|f-V_{2^{s-1}}(f)\|_{q,\mu}$. Note that the
series $\sum\limits_{j=s+1}^\infty\sigma_j(f)$ converges to
$f-V_{2^{s-1}}(f)$ in $L_{q,\mu}$ norm. Thus, by the Nikolskii
inequality \eqref{2.19} and \eqref{3.2}, we obtain
\begin{align}\label{3.4}
  \|f-V_{2^{s-1}}(f)\|_{q,\mu}&\lesssim \sum\limits_{j=s+1}^\infty\| \sigma_j(f)\|_{q,\mu}
   \nonumber\\&\lesssim \sum\limits_{j=s+1}^\infty (2^j)^{(d+2\mu)(\frac1p-\frac1q)_+}\| \sigma_j(f)\|_{p,\mu}
   \nonumber\\&\lesssim \sum\limits_{j=s+1}^\infty (2^{j})^{-r+(d+2\mu)(\frac1p-\frac1q)_+}\|f\|_{W_{p,\mu}^r}
  \nonumber\\&\lesssim (2^{s+1})^{-r+(d+2\mu)(\frac1p-\frac1q)_+}\|f\|_{W_{p,\mu}^r}
   \nonumber\\&\le n^{-r+(d+2\mu)(\frac1p-\frac1q)_+}\|f\|_{W_{p,\mu}^r}.
\end{align}
Next we estimate $\|L_{n,t}(f)-V_{2^{s-1}}(f)\|_{q,\mu}$. We note
that $(L_{n,t}(f)-V_{2^{s-1}}(f))\in \Pi_n^d$. It follows from
\eqref{2.19}, \eqref{1.1}, and \eqref{1.3} that
\begin{align}
  \|L_{n,t}(f)-V_{2^{s-1}}(f)\|_{q,\mu}&\lesssim n^{(d+2\mu)(\frac1t-\frac1q)_+}\|L_{n,t}(f)-V_{2^{s-1}}(f)\|_{t,\mu}
  \notag\\&\le A^{-1/t}n^{(d+2\mu)(\frac1t-\frac1q)_+}\|L_{n,t}(f)-V_{2^{s-1}}(f)\|_{(t)}
  \notag\\&\le A^{-1/t}n^{(d+2\mu)(\frac1t-\frac1q)_+}\{\|L_{n,t}(f)-f\|_{(t)}+\|f-V_{2^{s-1}}(f)\|_{(t)}\}
  \notag\\&\le 2 A^{-1/t}n^{(d+2\mu)(\frac1t-\frac1q)_+}\|f-V_{2^{s-1}}(f)\|_{(t)}
  \notag\\&\lesssim A^{-1/t}n^{(d+2\mu)(\frac1t-\frac1q)_+}\sum\limits_{j=s+1}^\infty\|
  \sigma_j(f)\|_{(t)}.\label{3.5}
\end{align}
According to \eqref{1.1}, it is easy to see that \eqref{2.9} is
true for $\{\tau_{n,k}\}_{k=1}^{l_n}$ with $p_0=t$ and $N=B$. Note
that $\sigma_j(f)\in\Pi_{2^{j}}^d$, and $2^j\ge n$ for $j\ge s+1$.
It follows from Lemma \ref{lem2.2}
 that  for $j\ge s+1$,
\begin{align*}
  \| \sigma_j(f)\|_{(t)}^t&=\sum\limits_{k=1}^{l_n}|\sigma_j(x_{n,k})|^t\tau_{n,k}\\&\lesssim B\Big(\frac{2^{j}}n\Big)^{d+2\mu}\|\sigma_j(f)\|_{t,\mu}^t
  \\&\lesssim B\Big(\frac{2^{j}}n\Big)^{d+2\mu}(2^j)^{t(d+2\mu)(\frac1p-\frac1t)_+}\|\sigma_j(f)\|_{p,\mu}^t
  \\&\lesssim
  B\Big(\frac{2^{j}}n\Big)^{d+2\mu}(2^{j})^{-rt+t(d+2\mu)(\frac1p-\frac1t)_+}\|f\|_{W_{p,\mu}^r}^t
  \\&= Bn^{-d-2\mu}(2^{j})^{-t\{r-(d+2\mu)[(\frac1p-\frac1t)_++\frac1t]\}}\|f\|_{W_{p,\mu}^r}^t,
\end{align*}that is,
\begin{equation}\label{3.6}
  \| \sigma_j(f)\|_{(t)}\lesssim B^{1/t}n^{-(d+2\mu)/t}(2^{j})^{-r+(d+2\mu)[(\frac1p-\frac1t)_++\frac1t]}\|f\|_{W_{p,\mu}^r}.
\end{equation}
Hence, for $r>(d+2\mu)\max\{1/p,1/t\}$, by \eqref{3.5} and
\eqref{3.3}  we obtain
\begin{align}\label{3.7}
  &\quad\|L_{n,t}(f)-V_{2^{s-1}}(f)\|_{q,\mu}\lesssim A^{-1/t}n^{(d+2\mu)(\frac1t-\frac1q)_+}\sum\limits_{j=s+1}^\infty\| \sigma_j(f)\|_{(t)}\nonumber
  \\&\lesssim \kappa^{1/t}n^{(d+2\mu)(\frac1t-\frac1q)_+} n^{-(d+2\mu)/t}
  \sum\limits_{j=s+1}^\infty (2^{j})^{-r+(d+2\mu)[(\frac1p-\frac1t)_++\frac1t]}\|f\|_{W_{p,\mu}^r}\nonumber
  \\&\lesssim\kappa^{1/t}n^{-r+(d+2\mu)[(\frac1p-\frac1t)_++(\frac1t-\frac1q)_+]}\|f\|_{W_{p,\mu}^r}.
\end{align}
By \eqref{3.3}, \eqref{3.4}, and  \eqref{3.7}, we get
\eqref{3.1}, which completes the proof of Lemma \ref{3.1}.

\end{proof}

\begin{rem}Let $1\le p\le\infty$,  $\mu\ge 0$, and  $r>(d+2\mu)/p$. Suppose that $(\mathcal{X},\tau)$ is an $L_{\infty}$-MZ family
with global condition number $\kappa=1/A$, $L_{n,\infty}$ is the
weighted least $\ell_\infty$ approximation  defined by
\eqref{1.2}. For $f\in W_{p,\mu}^r$, by \eqref{3.4} we have $$
\|f-V_{2^{s-1}}(f)\|_{(\infty)}\le
\|f-V_{2^{s-1}}(f)\|_\infty\lesssim
n^{-r+(d+2\mu)/p}\|f\|_{W_{p,\mu}^r}.
$$It follows that
$$
  \|f-L_{n,\infty}(f)\|_{\infty}\le C(1+\kappa)n^{-r+(d+2\mu)/p}\|f\|_{W_{p,\mu}^r},
$$where $C>0$  is independent of  $f$, $n$, $\kappa$, and
$(\mathcal{X},\tau)$.\end{rem}

\section{Weighted least $\ell_q$ approximation on the sphere}

In this section, we  discuss the weighted least $\ell_q$
approximation problem on the unit sphere $\Bbb S^d$ in
$\mathbb{R}^{d+1}$. Let   $L_p(\Bbb S^d),\ 0< p< \infty,$ denote
 the space of all Lebesgue measurable functions $f$ on
$\Bbb S^d$ with the finite quasi-norm
$$\|f\|_{p}:=\Big(\int_{\Bbb S^d}|f(x)|^pd\sigma(x)\Big)^{1/p},$$ where $d\sigma(x)$ is the rotationally invariant
measure on $\Bbb S^d$ normalized by $\int_{\Bbb S^d}d\sigma(x)=1$.
When $p=\infty$ we consider the space of continuous functions
$C(\Bbb S^{d})$ with the uniform norm.
 In particular, $L_2(\Bbb
S^d)$ is a Hilbert space with inner product $$\langle
f,g\rangle:=\int_{\Bbb B^d}f(x)g(x)d\sigma(x),\ {\rm for}\ f,g\in
L_2(\Bbb S^d).$$

We denote  by $\mathcal{H}_n(\Bbb S^d)$  the space of all
spherical harmonics of degree $n$, i.e., the space of the
restrictions to $\Bbb S^d$ of all  homogeneous harmonic
polynomials of exact degree $n$ on $\mathbb{R}^{d+1}$, and by
$\Pi_n(\Bbb S^d)$ the space of all spherical polynomials of degree
not exceeding $n$.

It is well known that the spaces $\mathcal{H}_n(\Bbb S^d),\,
n=0,1,2,\dots,$ are mutually orthogonal in $L_2(\Bbb S^d)$, and
 $$\triangle_0\, P=-n(n+d-1)P,\ \ {\rm for\ all}\ P\in\mathcal{H}_n(\Bbb S^d),$$where $\triangle_0$ is the Laplace-Beltrami operator on the sphere $\Bbb
 S^d$.
Let $$\{Y_{nk}\equiv Y_{nk}^d:\, k=1,2,\dots,b_n^d\}$$ be a fixed
orthonormal basis for $\mathcal{H}_n(\Bbb S^d)$, where $b_n^d={\rm
dim}\,\mathcal{H}_n(\Bbb S^d)$. Then $$\{Y_{nk}:\,
k=1,2,\dots,b_n^d,\, n=0,1,2,\dots\}$$ is an orthonormal basis for
$L_2(\Bbb S^d)$.

  The orthogonal projector $H_n :\,L_2(\Bbb S^d)\rightarrow \mathcal{H}_n(\Bbb S^d)$ can be written
 as\begin{align*}
     H_n f(x)&=\sum\limits_{k=1}^{b_n^d}\langle f, Y_{nk}\rangle Y_{nk}(x)
     =\langle f,E_n( x,\cdot)\rangle,
   \end{align*}
 where $E_n( x, y)=\sum\limits_{k=1}^{b_n^d}Y_{nk}(x)Y_{nk}(y)$ is the reproducing kernel of $\mathcal{H}_n(\Bbb S^d)$.
See \cite{DaX} for more details.

 Given $r>0$, define the fractional power $(-\triangle_0)^{r/2}$ of the operator $-\triangle_0$ on $f$ by
$$(-\triangle_0)^{r/2}(f):=\sum\limits_{k=0}^\infty (k(k+d-1))^{r/2}H_kf,$$in the sense of distribution.
Using this operator we define the Sobolev space on $\Bbb S^d$ as follows: for $r>0$ and $1\le p\le\infty$,
$$W_{p}^r(\Bbb S^d):=\{f\in L_{p}(\Bbb S^d):\|f\|_{W_{p}^r}:=\|f\|_{p}+\|(-\triangle_0)^{r/2}(f)\|_{p}<\infty\},$$
while the Sobolev class $BW_{p}^r(\Bbb S^d)$ is defined to be the
unit ball of the Sobolev space $W_{p}^r(\Bbb S^d)$. We remark
that $W_{2}^r(\Bbb S^d)$ is just the Sobolev space $H^r(\Bbb S^d)$
given in \cite{G}, and if $r>d/p$, then $W_{p}^r(\Bbb S^d)$ is
compactly embedded into $C(\Bbb S^d)$.

Similar to the case on $\BB$, we give the definitions of
$L_{q}$-Marcinkiewicz-Zygmund family and the weighted least
$\ell_q$ approximation on $\Bbb S^d$ as follows.

\begin{defn}\label{def4.1}
Suppose that
$\mathcal{X}=\{\mathcal{X}_n\}=\{x_{n,k}:\,k=1,2,\dots,l_n,\,n=1,2,\dots\}$
is a doubly-indexed set of points in $\Bbb S^d$, and
 $\tau=\{\tau_n\}=\{\tau_{n,k}:\,k=1,2,\dots,l_n,\,n=1,2,\dots\}$ is a doubly-indexed set of positive numbers.  Then for $0< q<\infty$, the family $(\mathcal{X},\tau)$ is
 called an $L_{q}$-Marcinkiewicz-Zygmund family on $\Bbb S^d$, denoted by $L_{q}$-MZ, if there exist constants $A,\, B >0$ independent
of $n$ such that
\begin{equation}\label{4.1}
  A\|P\|_{q}^q\le\sum\limits_{k=1}^{l_n}|P(x_{n,k})|^q\tau_{n,k}\le B\|P\|_{q}^q,\ {\rm for\ all}\ P\in \Pi_n(\Bbb S^d).
\end{equation}

The ratio $\kappa=B/A$ is the global condition number of
$L_{q}$-MZ family $(\mathcal{X},\tau)$, and
$\mathcal{X}_n=\{x_{n,k}:k=1,2,\dots,l_n\}$ is the $n$-th layer of
$\mathcal{X}$. Similarly, we can define  $L_\infty$-MZ family.
\end{defn}

\begin{rem}\label{rem4.2} Similar to the case on $\BB$, we set
\begin{equation*}
  \mu_n:=\sum\limits_{k=1}^{l_n}\tau_{n,k}\delta_{x_{n,k}}.
\end{equation*}
For any $f\in C(\Bbb S^d)$,  we define for $0<q<\infty$
$$\|f\|_{[q]}:=\Big(\int_{\Bbb
S^d}|f(x)|^qd\mu_n(x)\Big)^{1/q}=\Big(\sum\limits_{k=1}^{l_n}|f(x_{n,k})|^q\tau_{n,k}\Big)^{1/q}.$$and
for  $q=\infty$,  $$\|f\|_{[\infty]}:=\max\limits_{1\le k\le
l_n}|f(x_{n,k})|.$$ It follows from \eqref{4.1} that the
$L_{q}$-norm of a polynomial of degree at most $n$ on $\Bbb S^d$
is comparable to the discrete version given by the weighted
$\ell_q$-norm of its restriction to $\mathcal{X}_n$.  It follows
from \cite{BD,DaX,MO,MNW} that such MZ families exist if the
families are dense enough.
\end{rem}

\begin{defn}\label{def4.3}
Let $0< q\le\infty$, and let $(\mathcal{X},\tau)$ be an $L_{q}$-MZ
family on $\Bbb S^d$. For $f\in C(\Bbb S^d)$, we define the
weighted least $\ell_q$ approximation on $\Bbb S^d$ by
\begin{align}\label{4.2}
         L_{n,q}^{\Bbb S}(f)&:=\arg\min\limits_{P\in\Pi_n(\Bbb S^d)}\,\Big(\sum\limits_{k=1}^{l_n}|f(x_{n,k})-P(x_{n,k})|^q\tau_{n,k}\Big)^{1/q}.
 \end{align}That is,  $L_{n,q}^{\Bbb S}(f)$ is any function in $\Pi_n(\Bbb S^d)$ satisfying
 $$\|f-L_{n,q}^{\Bbb S}(f)\|_{[q]}=\min_{P\in\Pi_n(\Bbb S^d)}\|f-P\|_{[q]}.$$
\end{defn}

\begin{rem}\label{rem4.4}
Similar to the case on $\BB$, for $f\in C(\Bbb S^d)$ and $0<q\le
\infty$, the minimizer $L_{n,q}^{\Bbb S}(f)$ exists. Hence, this
definition is well defined. If $1< q<\infty$, $L_{n,q}^{\Bbb
S}(f)$ is unique. If $q=2$, then $L_{n,2}^{\Bbb S}(f)$ is linear.
However, if $0<q\le1$ or $q=\infty$, then $L_{n,q}^{\Bbb S}(f)$
may be not unique, and if  $q\neq 2$, then the operator
$L_{n,q}^{\Bbb S}$ is not linear.
\end{rem}

 For $L_{q}$-MZ family on $\Bbb S^d$ with $q=2$,
 $L_{n,2}^{\Bbb S}$ is a bounded linear operator on
$C(\Bbb S^d)$ satisfying that $(L_{n,2}^{\Bbb S})^2=L_{n,2}^{\Bbb
S}$, and the range of $L_{n,2}^{\Bbb S}$ is $\Pi_n(\Bbb S^d)$. If
we define the discretized inner product on $C(\Bbb S^d)$ by
$$\langle
f,g\rangle_{[2]}:=\sum\limits_{k=1}^{l_n}\tau_{n,k}f(x_{n,k})g(x_{n,k}),$$
then $L_{n,2}^{\Bbb S}$ is just the orthogonal projection onto
$\Pi_n(\Bbb S^d)$ with respect to the discretized inner product
$\langle \cdot,\cdot\rangle_{[2]}$. Hence, we get for $f\in C(\Bbb
S^d)$,
\begin{equation*}
  L_{n,2}^{\Bbb S}(f)(x)=\langle f,
  D_n^{\Bbb S}(x,\cdot)\rangle_{[2]}=\sum_{k=1}^{l_n}\tau_{n,k}f(x_{n,k})D_n^{\Bbb S}(x,x_{n,k}),
\end{equation*}where $D_n^{\Bbb S}(x,y)$ is the reproducing kernel of $\Pi_n(\Bbb S^d)$ with respect to the discretized inner product $\langle\cdot,\cdot\rangle_{[2]}$.
We call  $L_{n,2}^{\Bbb S}(f)$  the weighted least squares
polynomial on $\Bbb S^d$, and $L_{n,2}^{\Bbb S}$ the weighted
least squares operator on $\Bbb S^d$.

Following Gr\"ochenig  in \cite{G},  for $L_{2}$-MZ family  on
$\Bbb S^d$ we can also use the frame theory to construct the
quadrature formula
$$I_n^{\Bbb S}(f)=\sum_{k=1}^{l_n}W_{n,k}f(x_{n,k}).$$ It was shown in
\cite{LW} that
$$W_{n,k}=\tau_{n,k}\int_{\Bbb S^d}D_n^{\Bbb S}(x,x_{n,k})d\sigma(x), $$and
\begin{equation}\label{4.3}I^{\Bbb S}_n(f)=\int_{\Bbb S^d}L_{n,2}^{\Bbb S}(f)(x)d\sigma(x).\end{equation}Such
quadrature $I_n^{\Bbb S}$ is called the  least squares  quadrature on $\Bbb S^d$.

Analogous  to the case on $\BB$,  we obtain the following two
theorems. The proofs are   similar to the ones of Theorems
\ref{thm1.6} and \ref{thm1.7}.

\begin{thm}\label{thm4.5}Let $1\le p\le\infty$ and $1\le q\le\infty$. Suppose that $(\mathcal{X},\tau)$ is an $L_{q}$-MZ family on $\Bbb S^d$
with global condition number $\kappa$, $L_{n,q}^{\Bbb S}$ is the
weighted least $\ell_q$ approximation defined by \eqref{4.2}. For
$f\in W_{p}^r(\Bbb S^d)$, $r>d\,\max\{1/p,1/q\}$, we have for
$1\le q<\infty$
\begin{equation}\label{42.1}
  \|f-L_{n,q}^{\Bbb S}(f)\|_{q}\le C(1+\kappa^{1/q})n^{-r+d(\frac1p-\frac1q)_+}\|f\|_{W_{p}^r},
\end{equation}and for $q=\infty$,
$$
  \|f-L_{n,\infty}^{\Bbb S}(f)\|_{\infty}\le C(1+\kappa)n^{-r+d/p}\|f\|_{W_{p}^r},
$$
where $C>0$ depends on $r$, $d$, $p,\ q$,  but not on $f$, $n$,
$\kappa$ or $(\mathcal{X},\tau)$.
\end{thm}

\begin{thm}\label{thm4.6} Suppose that $(\mathcal{X},\tau)$
is an $L_{2}$-MZ family on $\Bbb S^d$ with global condition number
$\kappa$, $L_{n,2}^{\Bbb S}$ and $I_n^{\Bbb S}$ are the weighted
least squares approximation  and the  least squares quadrature,
respectively. If $f\in H^r(\Bbb S^d)\equiv W_2^r(\Bbb S^d)$,
$r>d/2$, then we have
\begin{equation}\label{4.9}
  \|f-L_{n,2}^{\Bbb S}(f)\|_{2}\le C(1+\kappa^{1/2})n^{-r}\|f\|_{H^r(\Bbb S^d)},
\end{equation}and
\begin{equation}\label{4.10}
  \Big|\int_{\Bbb S^d}f(x)d\sigma(x)-I_n^{\Bbb S}(f)\Big|\le C(1+\kappa^{1/2})n^{-r}\|f\|_{H^r(\Bbb S^d)},
\end{equation}
where $C>0$ depends on $r$, $d$, but not on $f$, $n$, $\kappa$
or $(\mathcal{X},\tau)$.
\end{thm}

Theorem \ref{thm4.5} is new, and Theorem \ref{thm4.6} is a slight
improvement of  \cite[Theorem 1.2]{LW}. Indeed,  we only reduce
dependence on the global condition number in \eqref{4.9} and
\eqref{4.10} by replacing the constant $(1+\kappa^2)^{1/2}$ in
\eqref{1.7-0} and \eqref{1.7-1} with the constant
$1+\kappa^{1/2}$.

\begin{rem}\label{rem4.8}
It follows from \cite{M,MO,MP} that there exist $L_{q}$-MZ
families on $\Bbb S^d$ with $l_n\asymp N\asymp n^d$. For such
$L_{q}$-MZ family, combining \eqref{42.1} with \cite[Theorem
1.2]{WW},  we obtain for $1\le p, q\le \infty$,
$r>d\,\max\{1/p,1/q\}$,
\begin{align*}\sup_{f\in BW_{p}^{r}(\Bbb S^d)}\|f-L_{n,q}^{\Bbb S}(f)\|_{q}\asymp N^{-\frac rd+(\frac1p-\frac1q)_+}\asymp g_N(BW_{p}^{r}(\Bbb S^d),L_{q}(\Bbb S^d)),\end{align*}
which implies that
 the weighted least $\ell_q$ approximation operators $L_{n,q}^{\Bbb S}$ are asymptotically optimal algorithms in the sense of optimal recovery for $1\le
p,q \le \infty$.

For the least squares quadrature rules $I_n^{\Bbb S}$, it follows from \eqref{4.10} and \cite{BH,H,WW} that for $r>d/2$,
$$ \sup\limits_{f\in BH^{r}(\Bbb S^d)}\Big|\int_{\Bbb S^d}f(x)d\sigma(x)-I_n^{\Bbb S}(f)\Big|\asymp N^{-\frac rd}\asymp e_N(BH^{r}(\Bbb S^d);{\rm INT}),$$which means that the  least squares quadrature rules $I_n^{\Bbb S}$ are the asymptotically optimal
quadrature formulas for $BH^{r}(\Bbb S^d)$.
\end{rem}

\vskip 3mm

\noindent{\bf Acknowledgment}  Jiansong Li and Heping Wang  were
supported by the National Natural Science Foundation of China
(Project no. 11671271).

\end{document}